\documentclass[11pt, a4paper]{amsart}

\usepackage{amsmath,amssymb,amscd,amsfonts, verbatim}
\def \seq {\sim_s}

\def \norm#1{\|#1\|}
\newcommand{\R}{{\mathbb R}}

\newcommand{\N}{{\mathbb N}}
\newcommand{\ol}{\overline}
\newcommand{\calH}{\mathcal H}
\newcommand{\calU}{\mathcal U}
\newcommand{\rk}{\operatorname {rk}}

\newtheorem{teo}{Theorem }[section]
\newtheorem{prop}[teo]{Proposition }

\newtheorem{notation}[teo]{Notation }
\newtheorem{lem}[teo]{Lemma }

\newtheorem{defi}[teo]{Definition }

\begin{document}
\title{Local algebraic approximation of semianalytic sets}

\author{M. Ferrarotti} 
\address{Dipartimento di Matematica\\Politecnico di Torino\\Corso Duca
degli Abruzzi 24\\I-10129 Torino, Italy}
\email{ferrarotti@polito.it}
\thanks{This research was partially supported by M.I.U.R. and by G.N.S.A.G.A}

\author{E. Fortuna}
\address{Dipartimento di Matematica\\Universit\`a di Pisa\\Largo
  B. Pontecorvo 5\\I-56127 Pisa, Italy}
\email{fortuna@dm.unipi.it}

\author{L. Wilson}
\address{Department of Mathematics\\University of Hawaii, Manoa\\Honolulu, HI 96822, USA}
\email{les@math.hawaii.edu}



\date{July 9, 2012}


\begin{abstract}
Two subanalytic subsets of $\R^n$ are called $s$-equivalent at a common point $P$ 
if the Hausdorff distance between their intersections with the sphere centered 
at $P$ of radius $r$ vanishes of order $>s$ when $r$ tends to $0$. In this paper 
we prove that  every $s$-equivalence class of a closed 
semianalytic set contains a semialgebraic 
representative of the same dimension. In other words any  semianalytic set can be 
locally approximated of any order $s$   by means of a semialgebraic set and hence, 
by previous results, also by means of an algebraic  one.
\end{abstract}

\maketitle

\section{Introduction}

In \cite{FFW-SNS} we introduced a notion of local metric proximity 
between two sets that we called $s$-equivalence: for a real $s\geq 1$, two 
subanalytic subsets of $\R^n$ are $s$-equivalent at a common point $P$ 
if the Hausdorff distance between their intersections with the sphere centered 
at $P$ of radius $r$ vanishes of order $>s$ when $r$ tends to $0$.

Given a subanalytic set $A\subset \R^n$ and a point $P\in A$, 
a natural question concerns the existence of an algebraic representative 
$X$ in the class of $s$-equivalence of $A$ at $P$; in that case we also say 
that $X$ approximates $A$ of order $s$ at $P$.

The answer to the previous question is in general negative for subanalytic sets which are not semianalytic, even for $s=1$ (see 
\cite{FFW-germs}).  Furtheremore, in  \cite{FFW-normal} we defined $s$-equivalence of two subanalytic sets along a common submanifold, and studied $1$-equivalence of a pair of strata to the normal cone of the pair.  By example we showed 
that a semianalytic normal cone to a linear $X$ may be not 1-equivalent to any semialgebraic set along $X$.
It is still an open problem whether a semialgebraic normal cone along a linear $X$ is $s$-equivalent to an
algebraic variety along $X$, for all $s$.   

On the other hand some partial positive answers were given in 
\cite{FFW-SNS} and \cite{FFW-germs}; in particular we proved that a subanalytic 
set $A\subset \R^n$  can be approximated of any order by an algebraic one in 
each of the following cases:

\noindent 
- $A$ is a closed semialgebraic set of positive codimension,\\
- $A$ is the zero-set $V(f)$ of a real analytic map $f$ whose regular points 
are dense in $V(f)$, \\
- $A$ is the image of a real analytic map $f$ having a finite fiber at $P$.

Using the previous results we also obtained that one-dimensional 
subanalytic sets, analytic surfaces in $\R^3$ and real analytic sets having 
a Puiseux-type parametrization admit an algebraic approximation of any order.

In the present paper we prove that any closed semianalytic set can be 
locally approximated of any order by a semialgebraic one having the same dimension. 
Using the main result of  \cite{FFW-SNS}, it follows that any closed semianalytic set 
of positive codimension admits an algebraic approximation of any order. Thus 
we obtain a complete positive answer to our question for the class of 
semianalytic sets. 

The algebraic approximation, elaborating the methods 
introduced in \cite{FFW-germs}, is obtained by taking sufficiently high order 
truncations of the analytic functions appearing in a presentation of the semianalytic 
set.

Finally, let us mention some possible future developments of these notions and ideas.
Since we can prove that two subanalytic sets $A,B$ are $1$-equivalent if 
and only if their tangent cones coincide (see also  \cite{FFW-SNS}), it would 
be interesting to extend the notion of tangent cone associating to $A$ a sort of
``tangent cone of order $s$'', say $C_s(A)$, in such a way that 
$A$ and $B$ are $s$-equivalent if and only if $C_s(A)=C_s(B)$.

There is currently a lot of interest in bilipschitz equivalence of varieties.  Most of the work has been in the 
complex case.  Two recent such examples are  \cite{BFGR} and  \cite{BFGO}.  The theory is closely tied up with
the notion of the tangent cone, exceptional subcones, and limits of tangent spaces.  The real case has been little
studied.  A good place to start is in the case of surfaces in $\R^3$, which is the only real case in which the
tangent cone, exceptional lines, and limits of tangent planes have been deeply analyzed (see  \cite{OW}).
The $s$-equivalence classes are Lipschitz invariants, so they should be a useful tool in this analysis.

\section{Basic notions and preliminary results}

If $A$ and $B$ are non-empty compact subsets of $\R^n$,   we denote
by $D(A,B)$ the classical Hausdorff distance, i. e.
$$
D(A,B)= \inf \ \{\epsilon\ |\ A\subseteq N_{\epsilon} (B), \
B\subseteq N_{\epsilon }
(A)\},
$$ 
where $N_{\epsilon } (A) =\{x\in \R^n\ |\ d(x,A)<\epsilon \}$ and 
$d(x,A)=\inf_{y\in A}\  \|x-y\|$.

If we let $\delta (A,B) = \sup _{x\in B} d(x,A) $, then $D(A,B)= \max
 \{\delta (A,B),\ \delta (B,A)\}$.
\smallskip

We will denote by $O$ the origin of $\R^n$ for any $n$.

We are going to introduce the notion of $s$-equivalence at a point; without loss 
of generality we can assume that this point is $O$.

\begin{defi} Let $A$ and $B$ be closed subanalytic subsets 
of $\R^n$ with $O\in A\cap B$. Let $s$ be a real number $\geq 1$. Denote by $S_r$ the 
sphere of radius $r$ centered at the origin. 
\begin{enumerate}
\item  We say that $A\leq_s B$  if either $O$ is isolated  in $A$, or if   
\item $O$ is non-isolated both in $A$ and in $B$ and 
$$\lim_{r\to 0}\frac{\delta(B \cap S_r,A\cap S_r)}{r^s} =0.$$ 
\item We say that $A$ and
$B$ are $s$--equivalent  (and we will write $A \sim_sB$) if 
$A\leq_s B$ and $B\leq_s A$. 
\end{enumerate}
\end{defi}

Observe that if  $O$ is non-isolated both in $A$ and in $B$, then  
$$A \sim_sB \qquad\mbox{ if and only if} \qquad  \displaystyle \lim_{r\to 0}\frac{D(A \cap S_r, B \cap S_r)}{r^s} =0.$$
Moreover, if $A\subseteq B$, then $A \leq_sB$ for any $s\geq 1$. 
It is easy to check that $\leq_s$ is transitive and that $\sim_s$ is an
equivalence relationship. The following result shows that $s$-equivalence  
has a good behavior  with respect to the union of sets:

\begin{prop}\label{union}\textup{(\cite{FFW-germs})}   Let $A$, $A'$, $B$ and $B'$ be closed
subanalytic subsets of $\R^n$.
\begin{enumerate}
\item If $A \leq_s B$ and $A' \leq_s B'$, then $A\cup A' \leq_s B\cup B'$.
\item If $A \sim_s B$ and $A' \sim_s B'$, then $A\cup A' \sim_s B\cup B'$.
\end{enumerate}
\end{prop}

Given a closed subanalytic set $A$ and $s \geq 1$, the problem we are 
interested in is whether there exists  an algebraic subset $Y$ which is
$s$-equivalent to $A$; in this case we also
say that $Y$ approximates $A$  to order $s$. Evidently the question is 
trivially true when  $O$ is an isolated point in $A$.

Among the partial answers to the previous question that have been already achieved, we  
recall only the following one which will be used later on:

\begin{teo}\label{approx-alg}\textup{(\cite{FFW-SNS})} For any real number
$s\geq 1$ and for any closed semialgebraic set $A \subset \R^n$ of
codimension $\geq 1$, there exists an algebraic subset
$Y$ of $\R^n$ such that $A \sim_sY$.
\end{teo}

The following definition introduces a geometric tool which is very useful to test the $s$-equivalence of two subanalytic sets:

\begin{defi}  Let $A$ be  a closed subanalytic subset
of $\R^n$, $O\in A$; for any real  $ \sigma>1$, 
we will  refer to the set
$$\mathcal H(A,\sigma)= \{ x\in \R^n \ |\  d(x,A)<  \|x\|^\sigma\}$$ 
as the horn-neighborhood with center $A$ and exponent $ \sigma$. 
\end{defi}

Note that, if $O$ is isolated in $A$, then $\mathcal H(A,\sigma)=\emptyset$ near $O$.

\begin{prop}\label{horn-lemma}\textup{(\cite{FFW-germs})} Let $A,B$ be closed subanalytic subsets
of $\R^n$ with $O\in A\cap B$ and let $s \geq 1$. Then
$A\leq_s B$ if and only if there exists $\sigma >s$ such that $A \setminus \{O\} \subseteq \mathcal H(B,\sigma)$.  
\end{prop}

The following technical result suggests that horn-neighborhoods can be used to modify 
a subanalytic set producing subanalytic sets $s$-equivalent to the original one:

\begin{lem}\label{XY} Let $X\subset Y \subset \R^n$ be closed subanalytic sets such that $O\in X$ and let $s\ge 1$. Then:
\begin{enumerate}
\item for any $\sigma >s$ we have $Y \seq Y \cup \calH (X,\sigma)$;
\item if $\ol{Y \setminus X} =Y$, there exists $\sigma >s$ such that 
$Y \setminus \calH (X,\sigma) \seq Y.$
\end{enumerate}
\end{lem}
\begin{proof} (a) Since $Y \cup \calH(X,\sigma)  \subseteq \calH (Y, \sigma)$, 
by Proposition \ref{horn-lemma} for any $\sigma >s$ we have that
$Y \cup \calH(X,\sigma) \leq_s Y$ and hence $Y \cup \calH (X,\sigma) \seq Y$.

(b) Let $\mathcal U (X,q )= \{x\in\R^n \ | \ \exists\, y\in X  ,\|x\|=\|y\|, 
\|x-y\| < \|x\|^q\}$.

Arguing as in \cite[Corollary 2.6]{FFW-SNS}, there exists $q$ such that 
$Y \setminus \calU (X,q) \seq Y$. Since $X$ and $Y \setminus \calU (X,q)$ 
are subanalytic sets  and meet only in $O$, they are regularly situated, 
i.e. there exists $\beta$ such that $d(x,X) + d(x, Y \setminus \calU (X,q)) > \|x\|^{\beta}$ 
for all $x$ near $O$. Then $\calH (X, \beta) \subseteq \calU (X,q)$ and hence taking 
$\sigma > \max\{\beta, s\}$ we have that $Y \setminus \calH (X,\sigma) \seq Y $.
\end{proof}

 Another essential tool will be \L ojasiewicz' inequality, which we will use in the following slightly modified version:

\begin{prop}\label{Loj} Let $A$ be a compact subanalytic subset of
$\R^n$. Assume $f$ and $g$ are subanalytic functions defined on $A$
such that $f$ is continuous, $V(f) \subseteq V(g)$, $g$  is continuous at 
the points of $V(g)$ and such that $|g|<1$ on $A$. Then there exists
a positive constant $\alpha$ such that $|g|^\alpha \leq |f|$ on $A$ and $|g|^\alpha <
|f|$ on $A\setminus V(f)$.
\end{prop}
\begin{proof} The result will be obtained by adapting the proof given by 
\L ojasiewicz under the stronger hypothesis that $g$ is continuous on $A$ 
(see \cite[Th\'eor\`eme 1]{L}); in that paper he used the following lemma (\cite[Lemma 4]{L}): 

 if $E\subset [0, \infty) \times \R$ is a compact semianalytic  subset 
of $\R^2$ such that $E\cap (\{0\}\times \R) \subseteq \{(0,0)\}$, then there exist positive 
constants $c, \alpha$ such that $E \subseteq \{(x,y)\in \R^2 \ |\ |y|^\alpha \leq c|x|\}$. 
\smallskip

The map $\Phi=(|f|,g)\colon A \to \R^2$ is subanalytic and bounded;  hence 
$\Phi(A)$ is a subanalytic subset of  $\R^2$ and therefore semianalytic (\cite[Proposition 2]{L}).
Then $E=\ol{\Phi(A)}$ is a compact semianalytic subset of $[0, \infty) \times \R$. 

We have that $E\cap (\{0\}\times \R) \subseteq \{(0,0)\}$: namely, if $(0,y_0)\in E$, 
then there exists a sequence $\{a_i\}\subset  A$ such that $\lim _{i\to \infty} \Phi(a_i)=(0,y_0)$ 
with $a_i$ converging to $a_0\in A$. By continuity $f(a_0)=0$ and hence $g(a_0)=0$. 
By the continuity of $g$ at $a_0$, we have that $y_0=g(a_0)=0$.

So $E$ fulfills the hypotheses of the lemma recalled above and therefore 
there exist  positive constants $c, \alpha$ such that $|g|^\alpha \leq c |f|$ on $A$.

Since  $|g|<1$, increasing $\alpha$ if necessary we can  obtain the thesis.
\end{proof}

\section{Main theorems}

This section is devoted to the proof of the local approximation theorem for semianalytic sets.

Since $s$-equivalence depends only on the set-germs at $O$, all the sets we will work 
with will be considered as subsets of a suitable open ball $\Omega$ centered at $O$; 
we will shrink such a ball whenever necessary without mention.

\begin{defi}
Let $A$ be a closed semianalytic subset of $\Omega$. We will say that 
$A$ admits a good presentation  if the minimal analytic variety  $V_A$containing $A$  
is irreducible and  there exist analytic functions $ f_1, \ldots, f_p$  
which generate  the ideal $I(V_A)$ 
  and $g_1, \dots , g_l$ analytic functions on $\Omega$ such that 
$$A=\{x \in \Omega \ |\ f(x)=O, g_i(x)\geq 0, i=1,\dots,l\}.$$
\end{defi}

We  start with a preliminary result concerning a way to decompose and present  semianalytic sets:

\begin{lem}\label{presentation} Let $A$ be a closed semianalytic subset of $\Omega$ with 
 $\dim_O A=d>0$.
Then there exist closed semianalytic sets $\Gamma_1, \ldots, \Gamma_r,\Gamma '$ such that
\begin{enumerate}
\item $A= \left(\bigcup_{i=1}^r \Gamma_i\right) \cup \Gamma '$
\item for each i, $\dim_O \Gamma_i =d$ and $\Gamma_i$ admits a good presentation
\item $\dim \Gamma ' < d$.
\end{enumerate}
\end{lem}
\begin{proof} Let $V_A$ be the minimal analytic variety containing $A$ (in particular $\dim_O V_A =d$). Let  
$V_1 \cup \ldots \cup V_m$ be the decomposition of $V_A$ into irreducible components. 
Then $A = W_1 \cup \ldots \cup W_m$ where $W_i = A \cap V_i$. Then $V_i$ is the minimal 
analytic variety containing $W_i$ and $\dim_O V_i = \dim_O W_i$. 

Each $W_i$ is a finite union of sets  of the kind  $\Gamma=\{h_1=0, \ldots, h_q=0, g_1 \geq 
0, \ldots, g_l\geq 0\}$.

Let $\Gamma '$ be the union, letting $i$ vary,  of the $\Gamma$'s having dimension less than $d$.
 
For any $\Gamma\subseteq V_i$ having dimension $d$, $V_i$ is the minimal analytic variety 
containing $\Gamma$.  It follows that $\Gamma=\{f_1=0,  \ldots, f_p=0, g_1 \geq 0, \ldots, g_l\geq 0\}$ 
where $f_1,  \ldots, f_p$ are generators of the ideal $I(V_i)$. Thus we can take as 
$\Gamma_1, \ldots, \Gamma_r$ these latter $\Gamma$'s (letting $i$ vary) suitably indexed.
\end{proof} 

\begin{notation}
Let $g_1, \dots , g_l$ be analytic functions on $\Omega$ and let $f= (f_1, \ldots, f_p) \colon \Omega \to \R^p$ 
be an analytic map. If $A=\{x \in \Omega\ |\ f(x)=O, g_i(x)\geq 0, i=1,\dots,l\},$ we will use the following notation:
\begin{enumerate}
\item $A_i =\{x\in \Omega \ |\  f(x)=O, g_i(x)\ge 0\}$ for $i=1, \dots, l$ \quad (so that $A=\bigcap A_i$)
\item $b(A)=\bigcup_{i=1}^l (V(g_i) \cap A)$.
\end{enumerate}
\end{notation}
\smallskip

\begin{lem}\label{Claim2} Consider the closed  semianalytic set 
$$A=\{x \in \Omega \ |\ 
f(x)=O, g_i(x)\geq 0, i=1,\dots,l\},$$ where $f\colon \Omega \to \R^p$ is an analytic map and 
$g_1, \dots , g_l$ are analytic functions on $\Omega$. Assume that $O\in A$. 
Let $\sigma$ be a real positive number and let $H\subseteq \R^n$ be an open subanalytic set such that 
$H\supseteq \mathcal H(b(A),\sigma)$. 
Then there exists $\eta$ such that, for each $x\in V(f) \setminus (A \cup H)$, 
there exists  $i$ so that $x\not\in \calH(A_i,\eta)$.
\end{lem}
\begin{proof}
Since the functions $\sum_{i} d(x,A_i)$ and $d(x,A)$ are subanalytic and vanish exactly on $A$, 
by Proposition \ref{Loj} there exists $\alpha >0$ such that, for any $x$,
$$\sum_{i} d(x,A_i) \geq d(x,A)^\alpha.$$

Let $d_g$ denote the geodesic distance on $V(f)$. 

If $x\in V(f) \setminus A$, we have  $d_g(x,A)=d_g (x,b(A))$. In a suitable closed 
ball centered at $O$ we can assume that $V(f)$ is connected; hence,  by a result of 
Kurdyka and Orro (\cite{KO}) for any $\epsilon >0$ there exists a subanalytic distance $\Delta (x,y)$ on $V(f)$ such that 
$$\forall x,y\in V(f) \qquad 0 \leq \Delta(x,y) \leq d_g(x,y) \leq (1+ \epsilon) \Delta(x,y).$$ 
Then, if we take for instance $\epsilon =1$, 
$$\forall x \in V(f) \qquad 0 \leq \Delta(x,A) \leq d_g(x,A) \leq 2 \Delta(x,A)  $$ 
and so 
the subanalytic function $\Delta(x,A)$ is continuous at each point of $A$. Hence  by 
Proposition \ref{Loj}  there exists $\mu >0$ such that, for any $x$ in $V(f)$,  
$$d(x,A) \geq \Delta(x,A)^\mu$$
and so 
$$\sum_{i} d(x,A_i) \geq \Delta(x,A)^{\mu \alpha}
\geq \left( \frac {d_g(x, A)}2 \right)^{\mu \alpha}.
$$
Moreover for any $x \in V(f) \setminus  (A\cup H)$ we have that 
$$d_g(x, A) = d_g(x, b(A)) \geq d(x,b(A)) \geq \|x\|^{\sigma }.$$

Let us show that the thesis holds choosing $\eta > \sigma \mu \alpha$.

If, for a contradiction, any neighborhood of $O$ contains a point  
$x\in \bigcap _i \calH(A_i,\eta)\cap (V(f) \setminus  (A\cup H)$, then  we have that
$$\frac1{2^{\mu \gamma}}\|x\|^{\sigma \mu \alpha} \leq \sum_{i=1}^l d(x,A_i)  \leq l \|x\|^\eta$$
which is impossible when $x$ tends to $O$.
\end{proof}

For any analytic map $\psi$ defined in a neighborhood of $O$, we will denote by 
$T^k \psi(x)$ the polynomial map whose components are the Taylor polynomials of
order $k$ at $O$ of the components of $\psi$.

\begin{lem}\label{Claim1} Let $\varphi$ an analytic function on $\Omega$ such that 
$\varphi(O)=0$. Let $X$ be a closed semianalytic subset of  $\Omega$, $O\in X$.  
Then for any real positive $\theta$ there exists $\alpha >0$ such that,  
for all  integers $k>\alpha$, the function $T^k \varphi$ has the same sign as  
$\varphi$ on $X \setminus \left( \calH(X \cap V(\varphi),\theta)\cup \{O\}\right)$.
\end{lem}
\begin{proof} Denote $Z= X \setminus \calH(X \cap V(\varphi),\theta)$. 
Since   $V(\varphi) \cap  Z = \{O\}$, by Proposition \ref{Loj} there exists $\alpha >0$ 
such that $\norm{x}^{\alpha} < |\varphi (x)|$ for all  $x\in Z\setminus \{O\}$.

For all  integers $k>\alpha$
$$\lim_{x\rightarrow O}\frac{\varphi (x) - T^k \varphi (x)} {\norm{x}^{\alpha}} =0.$$
 If $O$ is isolated in $Z$, there is nothing to prove. Otherwise assume, 
for a contradiction, that  any neighborhood of $O$ contains a point  $x\in Z$ such 
that $\varphi (x)$ and $T^k \varphi (x)$ have different signs (for instance $\varphi (x)>0$ and $T^k \varphi (x)\leq0$). Then 
$$\varphi (x) -T^k \varphi (x)\geq \varphi (x) > \|x\|^ \alpha$$
and hence $$\frac{\varphi (x) - T^k \varphi (x)} {\norm{x}^{\alpha}} >1$$ arbitrarily 
near to $O$, which is impossible.
\end{proof}

\begin{notation}
Let $g_1, \dots , g_l$ be analytic functions on $\Omega$ and let $f\colon \Omega \to \R^p$ 
be an analytic map. If $A=\{x \in \Omega\ |\ f(x)=O, g_i(x)\geq 0, i=1,\dots,l\},$ for any $h, k\in \N$ let 
\begin{enumerate}
\item $T^h(A)=\{x\in \Omega\ |\  T^hf(x)=O, g_i(x)\ge 0 \  i=1,\dots,l\}$
\item $T_k(A)=\{x\in \Omega\ |\  f(x)=O, T^kg_1(x)\ge 0, \ldots, T^kg_l(x) \ge 0  \}$
\item $T^h_k(A)=T^h(T_k(A))= \{x\in \Omega\ |\  T^hf(x)=O, T^kg_1(x)\ge 0, \ldots, T^kg_l(x) \ge 0 \}.$
\end{enumerate}
Moreover, for any analytic map $\varphi \colon \Omega \to \R^p$, denote  
$\Sigma_r(\varphi)=\{x\in \Omega \ |\ \rk\ d_x \varphi <r\}$, and $\Sigma(\varphi)=\Sigma_p(\varphi)$.
\end{notation}

\begin{lem}\label{trunc} 
Let $A$ be a closed semianalytic subset of $\Omega$, with $\dim_{O} A =d>0$. 
Assume that $A=\{f(x)=O, g_i(x)\geq 0, i=1,\dots,l\}$, with  $g_1, \dots , g_l$   analytic functions on $\Omega$ and $f\colon \Omega \to \R^{n-d}$ an analytic map. 
Assume also that $\dim_{O} (\Sigma (f)\cap A) <d$ and 
$\dim_{O} b(A) <d$. 
Then for any $s \geq 1$ there exist  $h_0 >0, k_0>0$ such that, for all integers $h,k$ with $h\geq h_0$ and $k \geq k_0$, we have 
\begin{enumerate}
\item $T^h_k(A) \leq_s A$
\item $\ol {A\setminus (\Sigma (f)\cup b(A))} \leq_s T^h_k(A)$
\item $\dim _O T^h_k(A) =d$.
\end{enumerate}
\end{lem}
\begin{proof} 
Let $s\geq 1$ and let $\sigma  >s$. 
Denote $X=(\Sigma (f) \cap A) \cup \,b(A) $.
\smallskip

 (1)   Let $H=\mathcal H(X,\sigma)$. 
By Lemma \ref{Claim2} there exists $\eta$ such that, for each 
$x\in V(f) \setminus (A \cup H)$, there exists  $i_0$ so that $x\not\in \calH(A_{i_0},\eta)$.

For all $j$, applying  Lemma \ref{Claim1} to $V(f)$, $ g_j$ and $\eta$, we find $\alpha_1 
>0$ such that, for all  integers $k>\alpha_1$, the functions
$g_j$ and $T^kg_j$ have the same sign on 
$ V(f) \setminus (\calH(V(f)\cap V(g_j),\eta) \cup \{O\})$.

 Let $x \in V(f) \setminus  (A \cup H)$.   Then  $x\not\in \calH(A_{i_0},\eta)$ 
for some $i_0$ and hence $g_{i_0}(x)<0$; moreover, since  $V(f)\cap V(g_{i_0})\subseteq  A_{i_0}$, 
we have that $x\in V(f) \setminus (\calH(V(f)\cap V(g_{i_0}),\eta) \cup \{O\})$ and hence,
  for all integers $k>\alpha_1$, $T^k g_{i_0} (x)<0$. This implies that $T_k(A) \subseteq  A\cup H$.

Applying Lemma \ref{XY} (1) to the sets $X$ and $A$, we have $ A \seq A\cup H$, 
and so $T_k(A) \leq_s A$.

Let $B_k=\{x\in \Omega \ |\ T^kg_i\geq 0, \ i=1, \ldots, l\}.$ 

 Since $T_k(A) = B_k \cap V(f)$, by Proposition \ref{Loj}  there exists  $\rho >0$ such that 
$\|f(x)\|\ge d(x,T_k(A))^\rho $ for all $x\in B_k$; then for $x \in B_k \setminus 
\mathcal H(T_k(A),\sigma)$ we have that $\|f(x)\|\geq \|x\|^{\rho \sigma}$. 

Let $h$ be an integer such that $h \geq \rho \sigma$. 
Then $$ \lim_{x\to O}
\frac{\|f(x)-T^hf(x)\|}{\|x\|^{\rho \sigma}}=0.$$

We have that $T^h(T_k(A))\setminus\{O\}
\subseteq \mathcal H(T_k(A),\sigma)$; otherwise there would exist a sequence of
points $y_i\ne O$ converging to $O$ such that $y_i\in T^h(T_k(A))\setminus
\mathcal H(T_k(A),\sigma)$ and hence
$$ \lim_{i \to \infty}\frac{\|f(y_i)-T^hf(y_i)\|}{\|y_i\|^{\rho \sigma}}=
\lim_{i \to \infty}\frac{\|f(y_i)\|}{\|y_i\|^{\rho \sigma}} \geq 1$$ which is
a contradiction. 

Then by Proposition \ref{horn-lemma} we get that
$T^h_k(A)\leq_s T_k(A) \leq_s A$.
\smallskip

(2) Let   $Y= \ol{A \setminus X}$.  By our hypotheses  $O$ is not isolated in $Y$.

Since $\ol{Y\setminus X} =Y$, applying Lemma \ref{XY} 
(2) to the sets $X\cap Y$ and $Y$, up to increasing 
 $\sigma$ we have that $ Y \setminus \mathcal H(X\cap Y ,\sigma)\seq Y$. 
 Denote 
$$Y'=Y \setminus \mathcal H(X\cap Y ,\sigma) \quad \mbox{ and } \quad 
H_i=\mathcal H(V(g_i) \cap Y,\sigma).$$
  If for each $i$ we  apply Lemma \ref{Claim1} to $Y$, $ g_i$  and $\sigma$, we can find 
$\alpha _2>0$ such that, for all  integers $k>\alpha_2$,  the functions 
$g_i$ and $T^kg_i$ have the same sign on  $Y\setminus (H_i\cup\{O\})$. 

Since $V(g_i) \cap Y \subseteq X \cap Y$ for each $i$,  then $\bigcup H_i \subseteq \mathcal H(X\cap Y ,\sigma)$, 
and therefore $Y' \setminus \{O\} \subseteq \bigcap_i (Y\setminus (H_i\cup\{O\}))$. 
In particular 
$$  Y' \setminus \{O\} \subseteq  \{T^kg_1>0, \ldots, T^kg_l>0\}.$$

From now on, assume that $k>\alpha_2$. We will get the result by replacing $f$ with a 
suitable truncation of it in the presentation of  $T_k(A)$. We will denote by $B(x,r)$ 
the open ball centered at $x$ of radius $r$.

By the last inclusion, the distance $d(x,b(B_k))$ is subanalytic and positive on 
$Y' \setminus\{O\}$ so, by  Proposition \ref{Loj}, there exists  $\nu >0$ 
(and we can assume $\nu > s$) such that $d(x,b(B_k))> \|x\|^{\nu}$ for all $x$ in
$Y'\setminus \{O\}$. As a consequence 
 $$B(x,\|x\|^{\nu})\subseteq \{T^kg_1>0, \ldots, T^kg_l>0\}.$$

Following \cite{FFW-germs} consider the real-valued function 
 $$\Lambda f(x)= \begin{cases} 0 &\text{if rk $d_x f<n-d$}\\
 \inf_{v\perp\ker{d_x f},\|v\|=1}\|d_x f(v)\|
 &\text{if rk $d_x f=n-d$}\end{cases}.$$

Observe that $\Lambda f(x)$ is subanalytic, continuous and  positive where $f$ 
is submersive, in particular on $Y' \setminus
\{O\}$. Hence, again by Proposition \ref{Loj}, there exists  $\beta >0$ 
such that $\Lambda f(x)> \|x\|^\beta$ for all $x$ in
$Y'\setminus \{O\}$.

Consider the subanalytic set 
$W=\{(x,y)\in Y' \times \Omega \ |\ \Lambda f(y)\geq \|x\|^\beta\}$ and let $W_0=
\{(x,y)\in Y' \times \Omega \ |\ \Lambda f(y)= \|x\|^\beta\}$; then the set $\{(x,x) \ |\ x\in
Y'\setminus \{O\}\} $ is contained in the open subanalytic set $ W \setminus W_0$.

The function $\varphi \colon Y' \setminus \{O\} \to \R$ defined by $\varphi
(x)= d((x,x),W_0)$ is subanalytic and positive. Then again
by Proposition \ref{Loj} there exists $\tau >0$ (and we can assume
$\tau >\nu$) such that $\varphi(x) > \|x\|^{\tau}$ on $Y'\setminus
\{O\}$. Then for all $x\in Y'\setminus \{O\}$ and for all $y \in B(x,
\|x\|^{\tau})$ we have
$$\|(x,y)-(x,x)\|=\|y-x\|< \|x\|^{\tau}< \varphi(x).$$ Hence $(x,y)\in
W\setminus W_0$, i.e. for all $x$ in $Y'\setminus \{O\}$ and for all $y \in
B(x, \|x\|^{\tau})$ we have $\Lambda f(y)> \|x\|^\beta.$ In particular
$\Lambda f(y)> 0$ and hence $d_yf$ is surjective for all $y \in B(x,
\|x\|^{\tau})$.
\smallskip

Let $h$ be an integer such that $h> \beta +1$ and let $\widetilde f(x)= T^hf(x)$. 

Then $T^{h-1}d_y f=d_y\widetilde f$; thus we have
that $\|d_y f-d_y \widetilde f\|\leq \|y\|^{h-1}$ 
for all $y$ near to $O$, where we consider 
$\operatorname {Hom} (\R^n,\R^{n-d})$ endowed with the standard norm
$$\|L\|= \max_{u\ne 0}
\frac{\|L(u)\|}{\|u\|}.$$

Thus by \cite[Proposition 3.3]{FFW-germs} we have 
$$|\Lambda f(y)- \Lambda \widetilde f(y)|\leq \|y\|^{h-1}.$$ 

{\bf Claim:} for $x\in Y'\setminus \{O\}$ and for $y \in B(x,\|x\|^{\tau})$, we have

$$\Lambda \widetilde f(y)\geq \|x\|^{\beta +1}.$$

To see this, assume for a contradiction that there exist a sequence
$x_i\in Y'\setminus\{O\}$ converging to $O$ and a sequence $y_i\in  B(x_i,
\|x_i\|^{\tau})$ such that $\Lambda \widetilde f(y_i)< \|x_i\|^{\beta +1}$. Thus we
have 
$$\frac{\Lambda f(y_i) - \Lambda \widetilde f(y_i)}{\|x_i\|^\beta}
> \frac{\|x_i\|^\beta - \|x_i\|^{\beta +1}}{\|x_i\|^\beta}= 1 - \|x_i\|.$$ 
On the other hand 
$$\frac{\Lambda f(y_i) - \Lambda \widetilde f(y_i)}{\|x_i\|^\beta}
\leq \frac{\|y_i\|^{h-1}}{\|x_i\|^\beta} \leq
\frac{(\|y_i-x_i\|+\|x_i\|)^{h-1}}{\|x_i\|^\beta} =$$
$$=
\left(\frac{\|y_i-x_i\|}{\|x_i\|^q}+\|x_i\|^{1-q}\right)^{h-1} \leq
\left( \|x_i\|^{\tau -q}+\|x_i\|^{1-q}\right)^{h-1}
$$
where $q=\frac{\beta}{h-1}$. Since $\tau> 1$ and $q<1$, we have that
$$\frac{\Lambda f(y_i) - \Lambda \widetilde f(y_i)}{\|x_i\|^\beta}$$ converges to
$0$, which is a contradiction. So the Claim is proved.
\smallskip

Then for all $x\in Y' \setminus \{O\}$ the map $\widetilde f$ is a
submersion on $B(x, \|x\|^{\tau})$. Hence, using \cite[Lemma 3.5]{FFW-germs}, we get
$\widetilde f(B(x,\|x\|^{\tau}))\supseteq B(\widetilde f(x), \|x\|^{\lambda})$ with
$\lambda=\beta+1+\tau$.

Observe that if $x\in Y'\setminus\{O\}$, we have that 
$$\lim_{x \to O}\frac{\|\widetilde f(x)\|}{\|x\|^h}= \lim_{x \to
O}\frac{\|\widetilde f(x)-f(x)\|}{\|x\|^h}=0.$$ 
So, for any $h\geq \lambda$ and $x\in Y'$,
the point $O$ belongs to $B(\widetilde f(x), \|x\|^{\lambda})$ and hence there exists
$y\in B(x,\|x\|^{\tau})$ such that $\widetilde f(y)=O$. 

Since $\tau >\nu >s$, then $y\in B(x,\|x\|^{\nu})$ so that $T^kg_i(y)>0$ for all $i$, i.e. $y \in T^h_k(A)$; hence
$Y' \setminus\{O\} \subseteq {\mathcal H}(T^h_k(A),\lambda)$. Then by Proposition 
\ref{horn-lemma}  we have $Y'\le_s T^h_k(A)$ 
and hence, since $Y'\seq Y$,we have that  
$$\ol {A\setminus (\Sigma (f)\cup b(A))}=Y \le_s T^h_k(A).$$

Therefore, taking  $h_0= \max\{ \rho \sigma, \lambda\}$ and $k_0 = \max\{\alpha_1, \alpha_2\}$, we have  the thesis.
\smallskip

(3) The previous argument shows that, for all $h\geq h_0$ and $k \geq k_0$,   
there exist points $y\in V(T^h f)$ arbitrarily 
near to $O$ where $T^h  f$ is submersive and such  that  $T^kg_i(y)>0$ for all $i$. 
Hence $\dim _O T^h_k(A) =d$. 
\end{proof}

\begin{teo}\label{general-semi} 
Let $A$ be a closed semianalytic subset of $\Omega$ with $O\in A$. 
Then for any $s \geq 1$ there exists a closed semialgebraic set $S\subseteq \Omega$ 
such that 
$A \sim_s S$ and $\dim_O S= \dim_O A$.
\end{teo}
\begin{proof} We will prove the thesis by induction on $d= \dim _O A$.

If  $d=0$ the result holds trivially. So let $d\geq 1$ and assume that the result holds for all semianalytic germs of dimension less that $d$.

By Lemma \ref{presentation}, by Proposition \ref{union} and by the inductive hypothesis, we can assume that 
$$A=\{x \in \Omega \ |\ f(x)=O, g_i(x)\geq 0, i=1,\dots,l\}$$ with $f=(f_1,\dots, f_p)$ such that $V(f)$ is irreducible, $V(f)$ is the minimal analytic variety containing $A$ and $f_1,\dots, f_p$ generate the ideal $I(V(f))$. In particular   $\dim _O(\Sigma_{n-d}(f) \cap A)<d$; moreover, removing from the previous presentation of $A$ the inequalities $g_i(x)\geq 0$ where $g_i$ vanishes identically on $A$ (if any), we can assume that $\dim _O b(A)<d$.

If $p=n-d$, the thesis follows easily by using Lemma \ref{trunc}.
 In general $p$ can be bigger than $n-d$; in this case we 
introduce a semianalytic set $\widetilde A$ of dimension $d$ which is $s$-equivalent to $A$ and which satisfies 
the hypotheses of Lemma \ref{trunc}. In order to prove the thesis it will be sufficient  to  approximate  $\widetilde A$  by means of a semialgebraic set having the same dimension. 

Denote by $\Pi$ the set of surjective linear maps from $\R^p$ to $\R^{n-d}$ and consider 
the smooth map $\Phi \colon (\R^n-V(f) ) \times \Pi \to \R^{n-d}$ defined by 
$\Phi (x, \pi) = (\pi \circ f) (x)$ for all $x\in \R^n-V(f) $ and $\pi \in \Pi$.

The map $\Phi$ is transverse to $\{O\}$: namely the partial Jacobian matrix of $\Phi$ with respect to the variables in $\Pi$ (considered as an open subset of 
$\R^{p(n-d)}$) is the $(n-d) \times p(n-d)$ matrix

$$ \bmatrix f(x)& 0&0 &\dots & 0\\ 0&f(x)&0&\dots &0  \\ \vdots \\0&0&0& \dots &f(x)\endbmatrix;$$
thus, for all $x\in \R^n-V(f)$ and for all $\pi \in \Pi$ the Jacobian matrix of
$\Phi$ has rank $n-d$. 

As a consequence, by a well-known result of singularity theory (see for instance \cite[Lemma 3.2]{Bruce-Kirk}), 
we have that the map $\Phi_{\pi} \colon  \R^n-V(f)   \to \R^{n-d}$ defined by 
$\Phi_{\pi} (x)= \Phi (x, \pi)=(\pi \circ f)(x)$ is transverse to $\{O\}$ for all $\pi$ outside a set $\Gamma \subset \Pi$ of measure zero and hence $\pi \circ f$ is a submersion on 
$V(\pi \circ f)\setminus V(f)$ for all such $\pi$.

Furthermore, let $x\in V(f)$ be a point at which $f$ has rank $n-d$; then there is an open dense set $U\subset \Pi$ such that  for all $\pi \in U$ the map $\pi \circ f$ is a submersion at  $x$, and hence off some subvariety of $V(f)$ of dimension less than $d$.

Thus, if we choose $\pi_0\in (\Pi\setminus \Gamma)\cap U$, the map $F= \pi_0 \circ f$ has $n-d$ components, $\Sigma(F) \cap V(F)\subseteq V(f)\subseteq V(F)$, $\dim_O V(F)=d$ and $\dim _O (\Sigma(F) \cap V(F))<d$. In particular $V(f)$ is an irreducible component of $V(F)$.

\smallskip

For each $m \in \N$ denote $\widetilde A_m=\{F=0,  \|x\|^{2m} - \|f\|^2 \geq 0, g_i(x)\geq 0, i=1,\dots,l\}$. 

Since $A \subseteq  \widetilde A_m \subseteq V(F)$, we have that $A \leq_s 
\widetilde A_m$ and $\dim _O \widetilde A_m =d$.

We claim that there exists $m$ such that $\widetilde A_m \sim_s A$; to  show that 
it is sufficient to prove that there exists $m$ such that $\widetilde A_m  \leq_s A$.
Namely, let $B=\{g_i(x)\geq 0, i=1,\dots,l\}$.  Since $V(\|f\|)\cap B=V(d(x,A) )\cap B $, by Proposition \ref{Loj}  there exists $q$ such that $d(x,A)^q \leq \|f(x)\|$ for all $x\in B$.
Let $m >  sq$. Then $d(x,A) \leq \|f(x)\|^{\frac 1q} \leq \|x\| ^ {\frac mq}$ for all $x\in \widetilde A_m$, i.e. $\widetilde A_m  \subseteq \mathcal H(A,\frac mq)$ and hence $\widetilde A_m  \leq_s A$. 

\medskip

Fix $m$ as above and let $\widetilde A =\widetilde A_m$.
Let also $\widetilde X= (\Sigma (F)\cap \widetilde A)\cup b(\widetilde A)$. 

Observe that $b(\widetilde A) \cap A = b(A)$ and so $\widetilde X\cap A= (\Sigma(F) \cap A) \cup b(A)$.

Denote $K= \widetilde X \cap (\widetilde A \setminus A)$ so that $\widetilde X = (\widetilde X\cap A) \cup K$.

By Lemma \ref{trunc} there exist positive integers $h, k$ such that 
$$\ol{\widetilde A\setminus \widetilde X}  \leq_s T^h_k(\widetilde A) \leq_s \widetilde A
\qquad \mbox{and} \qquad \dim_O T^h_k(\widetilde A)=d
.$$

Since $\dim_O (\widetilde X\cap A)<d$, by induction there exists a semialgebraic set $S_0$ such that $S_0 \seq \widetilde X\cap A$ and $\dim_O S_0 <d$. 
Moreover, since $A \subseteq \ol{\widetilde A \setminus K} \subseteq \widetilde A$, we have that 
$\ol{\widetilde A \setminus K} \sim_s \widetilde A$.

Then  
$$\widetilde A\sim_s \ol{\widetilde A \setminus K}=\ol{\widetilde A\setminus \widetilde X} \cup (\widetilde X\cap A) \leq_s T^h_k(\widetilde A) \cup S_0 \leq_s 
\widetilde A\cup (\widetilde X\cap A) = \widetilde A $$
so we can choose $S=T^h_k(\widetilde A) \cup S_0$. 
\end{proof}

From Theorem \ref{general-semi} and from Theorem \ref{approx-alg} we immediately obtain: 

\begin{teo}\label{general-alg} Let $A$ be a closed semianalytic subset of $\Omega$ of
codimension $\geq 1$ with $O\in A$. 
Then for any $s \geq 1$ there exists an algebraic set $Y \subset \R^n$ such that 
$A \sim_s Y$.

\end{teo}

\end{document}